\documentclass[11pt,leqno]{amsart}
\usepackage{amsmath,amsfonts,amssymb,amsthm}
\usepackage{accents}
\usepackage{enumitem}
\usepackage{constants}
\usepackage{xcolor}
\usepackage{tikz-cd}
\setlist[enumerate]{label={\rm(\roman*)}}

\newtheorem{theorem}{Theorem}[section]
\newtheorem{proposition}[theorem]{Proposition}

\theoremstyle{remark}
\newtheorem{remark}[theorem]{Remark}
\numberwithin{equation}{section}

\expandafter\let\expandafter\oldproof\csname\string\proof\endcsname
\let\oldendproof\endproof
\renewenvironment{proof}[1][\proofname]{%
  \oldproof[\bf #1]%
}{\oldendproof}

\def\M{\mathcal M}

\def\vr{\varrho}
\def\vp{\varphi}
\def\afi{A_{\vp}}

\begin{document}

\title{Weighted inequalities for sub-monotone functionals}

\author{Amiran Gogatishvili and Lubo\v s Pick}

\email[A.~Gogatishvili]{gogatish@math.cas.cz}
\urladdr{0000-0003-3459-0355}
\email[L.~Pick]{pick@karlin.mff.cuni.cz}
\urladdr{0000-0002-3584-1454}

\address{
 Institute of Mathematics of the Czech 
 Academy of Sciences,
 \v Zitn\'a~25,
 115~67 Praha~1,
 Czech Republic}

\address{
 Department of Mathematical Analysis,
	Faculty of Mathematics and Physics,
	Charles University,
	Sokolovsk\'a~83,
	186~75 Praha~8,
	Czech Republic}
	


% \thanks{This research was supported by grant 23-04720S of the Czech Science Foundation. The research of A. Gogatishvili was supported by the Czech Academy of Sciences (RVO: 67985840) and by  the Shota Rustaveli National Science Foundation of Georgia (SRNSFG), grant no: FR21-12353.}

\begin{abstract}
    We establish a set of relations between several quite diverse types of weighted inequalities involving various integral operators and fairly general quasinorm-like functionals which we call sub-monotone. The main result enables one to solve a specific problem by transferring it to another one for which a solution is known. The main result is formulated in a rather surprising generality, involving previously unknown cases, and it works even for some nonlinear operators such as the geometric or harmonic mean operators. Proofs use only elementary means.
\end{abstract}

\date{\today}

\maketitle

\section{Introduction and the main results}

Let $\mathcal M_+(0,\infty)$ be the collection of all measurable functions on $(0,\infty)$ with values in $[0,\infty]$. We say that a~functional $\vr\colon \mathcal M_+(0,\infty)\to [0,\infty]$ is \emph{sub-monotone} if it is monotone, weakly subadditive and subhomogeneous in the sense that one has
    \begin{equation}\label{E:lattice}
        \text{$\vr(f)\le\vr(g)$ for every $f,g\in \mathcal M_+(0,\infty)$ such that $f \le g$,}
    \end{equation}
    and there exists a~positive constant $K$ with the following two properties:
    \begin{equation}\label{E:quasitriangle}
        \text{if $\vr(\boldsymbol{1})<\infty$, then}\quad\vr(f+\boldsymbol{1}) \le K (\vr(f)+\vr(\boldsymbol{1}))\quad\text{for every $f\in \mathcal M_+(0,\infty)$,}
    \end{equation}
    where $\boldsymbol{1}=\chi_{(0,\infty)}$, and
    \begin{equation}\label{E:weak-lattice}
        \vr(\lambda f)\le K\lambda\vr(f)
        \quad\text{for every $\lambda\ge0$ and $f\in \mathcal M_+(0,\infty)$.}
    \end{equation}
    
    Note that a sub-monotone functional is defined by axiomatized properties which are essentially less restrictive than those that define a~monotone quasinorm (see~\cite{Gog:20}).
    
A wide variety of customary functionals are sub-monotone. An elementary example is the (quasi-)norm in the weighted Lebesgue space, namely
\begin{equation*}
    \varrho(f) = \left\|f\right\|_{L^q(v)}
\end{equation*}
for some fixed $q\in(0,\infty]$ and given \emph{weight} (measurable and almost everywhere strictly positive function on $(0,\infty)$) $v$, but one can also bring in some more exotic ones such as
\begin{equation*}
    \varrho(f) = \left\|\sup_{s\in(t,\infty)}f(s)v(s)\right\|,
\end{equation*}
in which $\|\cdot\|$ is an arbitrary quasinorm, or
\begin{equation*}
    \varrho(f) = \left\|\left(\int_{0}^{t}f^rw\right)^{\frac{1}{r}}\right\|_{L^{q}(v)}
\end{equation*}
with $r\in(0,\infty)$ and another weight $w$, and so on.

Given a weight $v$, we will throughout denote by $V$ the function given by
\begin{equation*}
    V(t)=\int_{0}^{t}v(s)\,ds
    \quad\text{for $t\in(0,\infty)$.}
\end{equation*}

A prototype of the problems which we focus on in this paper is the following: given $p\in(1,\infty)$, a weight $v$ and a sub-monotone functional $\varrho$, does there exist a positive constant $C$ such that
\begin{equation}\label{E:basic}    
    \vr\left(\frac{1}{V(t)}\int_{0}^{t}f(s)v(s)\,ds\right) \le C \left(\int_{0}^{\infty}f(t)^{p}v(t)\,dt\right)^{\frac{1}{p}}
    \quad\text{for every $f\in\M_+(0,\infty)$},
\end{equation}
and, if so, is there a quantification of the optimal such $C$ in terms of the parameters involved. It is worth noticing that the requirements on $\varrho$ are  very mild (it does not even have to be a quasinorm), hence there is a considerable versatility of important inequalities to which this approach can be applied. In particular, many customary inequalities can be cast in the form~\eqref{E:basic}. Consider, for instance, the classical two-weight Hardy inequality
\begin{equation}\label{E:hardy}
    \left(\int_{0}^{\infty}\left(\int_{0}^{t}f\right)^{q}w(t)\,dt\right)^{\frac{1}{q}}
        \leq C
        \left(\int_{0}^{\infty}f^{p}u\right)^{\frac{1}{p}}
\end{equation}
with $p\in(1,\infty)$ and $q\in(0,\infty)$. Then,~\eqref{E:hardy} becomes~\eqref{E:basic} on setting first $v=u^{1-p'}$, and then
\begin{equation*}
    \varrho(f)=\left(\int_{0}^{\infty}f^qwV^q\right)^{\frac{1}{q}}
    \quad\text{for $f\in\M_+(0,\infty)$.}
\end{equation*}

The requirement $p\in(1,\infty)$ is quite natural since, for any $p\in(0,1)$, one can easily construct functions for which
\begin{equation*}
     \int_{0}^{\infty}f^{p}u
     <\infty,
\end{equation*}
but which are not locally integrable, whence the inequality~\eqref{E:hardy} would be impossible. The remaining case $p=1$ requires a separate treatment which is not considered here.

The question of characterizing pairs of weights for which~\eqref{E:hardy} holds is one of the oldest in analysis, and a lot of effort has been spent on finding the answer. For $p=q>1$, $v=1$,
$w(t)=t^{-q}$, it collapses to the boundedness of the integral averaging operator on $L^{p}(0,\infty)$, see~\cite{HLP:88}. The cases with general weights were probably first studied by Kac and Krein~\cite{KacKre:58} for $p=q=2$ and $v=1$, then by Beesack, see e.g.~\cite{Bee:61} for some other specific weights, by Tomaselli~\cite{Tom:69}, Talenti~\cite{Tal:69} and Muckenhoupt~\cite{Muc:72} for $p=q$, by Bradley~\cite{Bra:78}, see also (without proof) Kokilashvili~\cite{Kok:79} for $p\leq q$, and literature records also some unpublished papers, one by Artola, and another one by Boyd and Erd\H{o}s (\cite{Rie:74}). The case $1\le q<p<\infty$ was first characterized by Maz'ya and Rozin (published later in the book~\cite{Maz:11}) and Sawyer~\cite{Saw:84}, the case $0<q<1<p<\infty$ by Sinnamon~\cite{Sin:91}, and the case $0<q<p=1$ by Sinnamon and Stepanov~\cite{SiSt:96}, see also~\cite{BeGr:06}. 

Inequality~\eqref{E:hardy} is often interpreted as the boundedness of the \emph{Hardy operator} $f\mapsto\int_0^tf$ from one weighted Lebesgue space into another. In a similar fashion, other operators have been considered. We can recall, for instance, the \emph{Copson operator} $f\mapsto\int_t^{\infty}f$, the \emph{geometric mean operator} $f\mapsto\exp(\frac1t\int_t^{\infty}\log f)$,
or the \emph{harmonic mean operator} $f\mapsto\ t(\int_0^t\frac1f)^{-1}$. The latter two operators are notably nonlinear, and, moreover, they constitute particular cases of a general operator $\afi$, defined with the help of a strictly monotone function $\vp\colon(0,\infty)\to(0,\infty)$, which is either concave and increasing (such as $\log$), or convex and decreasing (such as $t\mapsto \frac1t$), by
\begin{equation*}
    \afi f(t)=\vp^{-1}\left(\frac{1}{t}\int_{0}^{t}\varphi(f(s))\,ds\right).
\end{equation*} 
Further modifications such as inequalities involving superposition of operators, iterated operators, operators involving kernels and various their combinations were studied separately. For instance, weighted inequalities for the geometric mean operator were characterized e.g.~in~\cite{Pi-Op:94}. Many variants of this operator have been extensively studied, see e.g.~\cite{Ja-Si:00}. Weighted inequalities involving this operator are known to be very important. In the literature they are sometimes called the \textit{Carleman-Knopp inequalities}, as in~\cite{Ja-Pe-We:01}. The operator is known to be in intimate relation with basic inequalities such as the arithmetic-geometric mean inequalities and their applications in approximation theory, see e.g.~\cite{Ka-Pe-Ob:02} or~\cite{Pe-St:02}.

Many authors spent considerable effort in order to characterize boundedness of each of the mentioned operators separately, grappling with their intrinsic technicalities, and applying quite varying approaches in dependence on the parameters involved. While, during many decades of heavy investigation, plenty of information was found about each individual operator, and also about each individual case, very little has been discovered about the interplay between them. To establish certain knowledge of this sort is one of the principal goals of this paper.  

There is one more important feature of our approach that is worth noticing. As far as weighted inequalities such as~\eqref{E:hardy} are concerned, there has always been a general feeling that the cases $0<p\le q<\infty$, $1\le q<p<\infty$, $0<q<p=1$, and $0<q<p<1$, are substantially different from one another, and have to be treated separately.  For example, while only elementary integration and H\"older's and Minkowski's inequalities are required for the case $0<p\le q<\infty$, deep techniques such as the Halperin level function have been used for the case $0<q<p<1$, and, in order to handle yet other cases, variants of duality methods, or the discretization and antidiscretization techniques were brought in. A notable exception is the approach applied in~\cite{SiSt:96}, where a~universal technique is developed for several cases.

Our mission in this paper is different. We take a lateral point of view on the inequalities, one of its principal achievements being the rather surprising observation that all the mentioned cases are linked by elementary transformations, and as such can be easily reduced to one another. Inequalities for Hardy, Copson, geometric mean and harmonic mean operators are shown to be interlinked. No involved techniques such as level function, duality, or discretization, are used. 

We can now state our main result.

\begin{theorem}\label{T:1}
    Assume that $\vr\colon \mathcal M_+(0,\infty)\to [0,\infty]$ is a sub-monotone functional.
    Let $p\in(1,\infty)$ and let $v$ be a weight. Then the following statements are equivalent:
    
    \textup{(i)} there exists a positive constant $C_1$ such that
    \begin{equation}\label{E:1}    
        \vr\left(\frac{1}{V(t)}\int_{0}^{t}f(s)v(s)\,ds\right) \le C_1 \left(\int_{0}^{\infty}f(t)^{p}v(t)\,dt\right)^{\frac{1}{p}}
    \end{equation}
    for every $f\in\M_+(0,\infty)$,
    
    \textup{(ii)} for every $r\in[1,\infty)$ and $\alpha\in (\max\{-\frac 1p, -\frac1{p'}\},\infty)$ there exist positive constants $C_{2,1}$ and $C_{2,2}$ such that
    \begin{equation}\label{E:2}
        \vr\left(
        V(t)^{\alpha }
        \left(\int_{t}^{\infty}f(s)\frac{v(s)}{V(s)}\,ds\right)^{\frac{r}{p}}\right)
        \le C_{2,1} 
        \left(\int_{0}^{\infty}f(t)^{r}V(t)^{\alpha p}v(t)\,dt\right)^{\frac{1}{p}}
    \end{equation}
    for every $f\in\M_+(0,\infty)$,
    and
    \begin{equation}\label{E:constant}
        \varrho(\boldsymbol{1})\le C_{2,2} \left(\int_{0}^{\infty}v(t)\,dt\right)^{\frac{1}{p}},
    \end{equation}
    
    \textup{(iii)} there exist $r\in[1,\infty)$, 
    $\alpha\in (\max\{-\frac 1p, -\frac1{p'}\},\infty)$, and positive constants $C_3$ and $C_{2,2}$ such that
    \begin{equation}\label{E:2-bis}
    \vr\left(V(t)^{\alpha}\left(\int_{t}^{\infty}f(s)\frac{v(s)}{V(s)}\,ds\right)^{\frac{r}{p}}\right)
            \le C_3 \left(\int_{0}^{\infty}f(t)^{r}V(t)^{\alpha p }v(t)\,dt\right)^{\frac{1}{p}}
    \end{equation}
    for every $f\in\M_+(0,\infty)$ and \eqref{E:constant} is satisfied.
    
    \textup{(iv)} for every triple $(r,\alpha,\beta)$ satisfying $r\in[1,\infty)$, 
    $\alpha\in (\max\{-\frac 1p,-\frac1{p'}\},\infty)$, 
    $\beta\in(-\frac{1}{r'},\infty)$,
    \begin{equation}\label{E:condition-iv}
        \alpha p-\beta r<r-1
    \end{equation}
    there exists a positive constant $C_4$ such that
    \begin{equation}\label{E:3}
    \vr\left(\left(\frac{1}{V(t)}\int_{0}^{t}f(s)v(s)\,ds\right)^{\frac{r}{p}}V(t)^{\alpha-\frac{\beta r}{p}} \right)
            \le C_4 \left(\int_{0}^{\infty}f(t)^{r}V(t)^{\alpha p-\beta r}v(t)\,dt\right)^{\frac{1}{p}}
    \end{equation}
    for every $f\in\M_+(0,\infty)$,
    
    \textup{(v)} there exist a~triple $(r,\alpha,\beta)$ satisfying $r\in[1,\infty)$, 
    $\alpha\in (\max\{-\frac 1p,-\frac1{p'}\},\infty)$, 
    $\beta\in(-\frac{1}{r'},\infty)$,  and such that~\eqref{E:condition-iv} holds, and a positive constant $C_5$ such that
    \begin{equation}\label{E:3-bis}
        \vr\left(\left(\frac{1}{V(t)}\int_{0}^{t}f(s)v(s)\,ds\right)^{\frac{r}{p}}V(t)^{\alpha - \frac{\beta r}{p}} \right)
            \le C_5 \left(\int_{0}^{\infty}f(t)^{r}V(t)^{\alpha p-\beta r}v(t)\,dt\right)^{\frac{1}{p}}
    \end{equation}
    for every $f\in\M_+(0,\infty)$,
    
    \textup{(vi)} there exists a positive constant $C_6$ such that
    \begin{equation}\label{E:4}
        \vr\left(\exp\left( \frac{1}{V(t)}\int_{0}^{t}\log(f(s))v(s)\,ds\right)\right)
        \le C_6
        \left(\int_{0}^{\infty}f(t)^{p}v(t)\,dt\right)^{\frac{1}{p}}
    \end{equation}
    for every measurable strictly positive function $f$ on $(0,\infty)$,
    
    \textup{(vii)} for every $r\in(0,\infty)$  there exist a positive constant $C_7$ such that
    \begin{equation}\label{E:5}
     \vr\left(\left(\frac{V(t)}{\int_{0}^{t}f(s)^{-1}v(s)\,ds}\right)^r\right)
      \le C_7 \left(\int_{0}^{\infty}f(t)^{rp}v(t)\,dt\right)^{\frac{1}{p}}
    \end{equation}
    for every measurable strictly positive function $f$ on $(0,\infty)$,
    
    \textup{(viii)} there exist $r\in(0,\infty)$  and  a positive constant $C_8$ such that
    \begin{equation}\label{E:5-bis}
     \vr\left(\left(\frac{V(t)}{\int_{0}^{t}f(s)^{-1}v(s)\,ds}\right)^r\right)
      \le C_8 \left(\int_{0}^{\infty}f(t)^{rp}v(t)\,dt\right)^{\frac{1}{p}}
    \end{equation}
    for every measurable strictly positive function $f$ on $(0,\infty)$,
    
    \textup{(ix)} Fix a function $\varphi\colon (0,\infty) \to (-\infty,\infty)$ that is either concave and non-decreasing or else
    convex and non-increasing. There exists  a positive constant $C_9$ such that
    \begin{equation}\label{E:9-bis}
    \vr\left(\varphi^{-1}\left(\frac{1}{V(t)}\int_{0}^{t}\varphi(f(s))v(s)\,ds\right)
     \right)
     \le C_9 \left(\int_{0}^{\infty}f(t)^{p}v(t)\,dt\right)^{\frac{1}{p}}
    \end{equation}
    for every measurable strictly positive function $f$ on $(0,\infty)$.
\end{theorem}

\begin{remark}
    The proof of Theorem~\ref{T:1} will be given in Section~\ref{S:proofs}, and it is organized as follows. We will first establish a cobweb of implications pouring down from (i) to (iii) through each of the remaining ones, more precisely,
    \begin{align*}
        &\textup{(i)}
        \Rightarrow\textup{(vi)}
        \Rightarrow\textup{(ii)}
        \Rightarrow\textup{(iii)},
                    \\
        &\textup{(ii)}
        \Rightarrow\textup{(iv)}
        \Rightarrow\textup{(v)}
        \Rightarrow\textup{(iii)},    
                \\
        &\textup{(vi)}
        \Rightarrow\textup{(vii)}
        \Rightarrow\textup{(viii)}
        \Rightarrow\textup{(iii)},    
            \\
        &\textup{(i)}
        \Rightarrow\textup{(ix)}
        \Rightarrow\textup{(iii)}. 
            \intertext{Finally, we will round off the picture by showing that}
        &\textup{(iii)}
        \Rightarrow\textup{(i)}   . 
    \end{align*}
    
    Some observations are due. There are certainly numerous other possibilities of how to organize the proof. However, one of our principal goals is not only to establish the result itself, but also to point out interconnections between separate statements (i)--(ix) of Theorem~\ref{T:1}. Therefore, for example, we do not treat statements (vi), (vii) and (viii) as particular cases of (ix) (even though they clearly are) because, by avoiding that, we spare a reader, who might be possibly interested for instance in equivalence of inequality for the geometric mean operator to that for a Hardy or Copson operator, resorting necessarily to a statement like (ix) that involves some general function $\varphi$. Moreover, our approach to the proof does not conceal the short direct arguments that can be used to obtain equivalences of these separate statements, which are certainly of independent interest. 

    In the scheme given above, the implication \textup{(ii)}
        $\Rightarrow$\textup{(iii)} is obviously redundant. However, since this implication is completely trivial as it follows just from the quantification of relevant parameters (similarly as the implications \textup{(iv)}
        $\Rightarrow$\textup{(v)} and \textup{(vii)}
        $\Rightarrow$\textup{(viii)}), it does not really matter.

\end{remark}

\begin{remark}
    There are several applications of Theorem~\ref{T:1} that should not be missed. First, it shows that the results of Sinnamon~\cite{Sin:91}, Maz'ya--Rozin~\cite{Maz:11} and Sinnamon--Stepanov~\cite{SiSt:96} can be easily reduced one to another. Next, the result shows that the inequalities involving nonlinear operators such as the geometric and harmonic mean operators can be, once again, reduced to classical Hardy inequalities. Last but not least, Hardy and Copson inequalities are reducible to one another.
    
    Roughly speaking, once we know a characterization of some (easy, say) inequality, such as~\eqref{E:hardy} for $p\le q$, we can, by a simple application of Theorem~\ref{T:1}, obtain a characterization of a seemingly much more difficult inequality, involving either a different relation between parameters, or some other operator, or both.
\end{remark}

We will present two more theorems, which can be understood as characterizations of inequalities for the geometric mean operator, and for the harmonic mean operator, in terms of ordinary Hardy inequalities, for which it is easy to single out the class of pairs of weight for which these inequalities hold.

\begin{theorem}\label{T:3} Assume that $\mathcal M_+(0,\infty)$ is a partially ordered vector space and $\vr\colon \mathcal M_+(0,\infty)\to [0,\infty]$ is a sub-monotone functional.
Let $p\in(0,\infty)$, and let $v, u$ be weights. Denote by $U(t)=\int_{0}^{t}u(s)\,ds$ and  
\begin{equation*}
    w(t)=\exp\left(\frac{1}{p U(t)}\int_{0}^{t}\log\left(\frac{u(s)}{v(s)}\right)u(s)\,ds\right)
\end{equation*} for $t\in(0,\infty)$.

Then the following statements are equivalent:

\textup{(i)} there exists a positive constant $C_{10}$ such that
    \begin{equation}\label{E:10}    
        \vr\left(\exp\left( \frac{1}{U(t)}\int_{0}^{t}\log(f(s))u(s)\,ds\right)\right)
        \le C_{10} \left(\int_{0}^{\infty}f(t)^{p}v(t)\,dt\right)^{\frac{1}{p}}
    \end{equation}
    for every $f\in\M_+(0,\infty)$,

\textup{(ii)} let $r\in [1,\infty)$, $m\in[\frac{1}{p},\infty)$  and $\alpha\in (\max\{-\frac{1}{mp}, -\frac1{(mp)'}\},\infty)$,  then there exists a positive constant $C_{11}$ such that
    \begin{equation*}
       \vr \left(\left(\int_{t}^{\infty}f(s)\frac{u(s)\,ds}{U(s)}\right)^{\frac rp}U(t)^{\alpha m}w(t)\right)
                \le C_{11} \left(\int_{0}^{\infty}f(t)^{r}U(t)^{\alpha mp}u(t)\,dt\right)^{\frac{1}{p}} 
    \end{equation*}
    for every $f\in\M_+(0,\infty)$,

\textup{(iii)}  if $r\in [1,\infty)$,  $m\in (\frac{1}{p},\infty)$  and $\alpha\in (\max\{-\frac{1}{mp}, -\frac1{(mp)'}\},\infty)$,  $\beta\in(-1,\infty)$, and  such that 
$$\alpha mp - \beta r< r-1,$$ 
then there exists a positive constant $C_{12}$ such that
    \begin{equation*}  
    \vr\left(\left(\frac{1}{U(t)}\int_{0}^{t}f(s)v(s)\,ds\right)^{\frac{r}{p}}U(t)^{\alpha m-\frac{\beta r}{p}} w(t)\right)
        \le C_{12} \left(\int_{0}^{\infty}f(t)^{r}U(t)^{\alpha p m-\beta r}u(t)\,dt\right)^{\frac{1}{p}}     \end{equation*}
    for every $f\in\M_+(0,\infty)$.
\end{theorem}

\begin{theorem}\label{T:4}
Let $q,p, m\in(0,\infty)$, and $mp>1$, $\alpha, \beta \in (-1,\infty)$ and let $v,w, u$ be weights. Denote by $U(t)=\int_{0}^{s}u(s)\,ds$, $\widetilde{u}(t)=u(t)^{\frac{p}{p+1}}v(t)^{\frac{1}{p+1}}$ and $\widetilde{U}(t)=\int_0^t  u(s)^{\frac{p}{p+1}}v(s)^{\frac{1}{p+1}}\,ds$. Then the following statements are equivalent:

\textup{(i)} there exists a positive constant $C_{13}$ such that
    \begin{equation}\label{E:13}    
        \vr\left(\frac{U(t)}{\int_{0}^{t}f(s)^{-1}u(s)\,ds}\right)
        \le C_{13} \left(\int_{0}^{\infty}f(t)^{p}v(t)\,dt\right)^{\frac{1}{p}} 
    \end{equation}
    for every $f\in\M_+(0,\infty)$,

\textup{(ii)} there exists a positive constant $C_{14}$ such that
    \begin{equation*}  
        \vr\left(\left(\int_{t}^{\infty}f(t)\frac{\widetilde{u}(t)\,dt}{\widetilde{U}(t)}\right)^{m}
        U(t)\widetilde{U}(t)^{\alpha m-1}\right)
        \le C_{14} \left(\int_{0}^{\infty}f(t)^{mp}\widetilde{U}(t)^{\alpha  mp}\widetilde{u}(t)\,dt\right)^{\frac{1}{p}}
    \end{equation*}
    for every $f\in\M_+(0,\infty)$,

\textup{(iii)} there exists a positive constant $C_{15}$ such that
    \begin{equation*}   
         \vr \left(\left(\frac{1}{\widetilde{U}(t)}\int_{0}^{t}f(t)\widetilde{u}(t)\,dt\right)^{m}
          U(t)\widetilde{U}(t)^{-(\beta -\alpha)m-1}\right)
        \le C_{15} \left(\int_{0}^{\infty}f(t)^{mp}\widetilde{U}(t)^{-(\beta -\alpha)mp}  \widetilde{u}(t)\,dt\right)^{\frac{1}{p}} 
    \end{equation*}
    for every $f\in\M_+(0,\infty)$.
\end{theorem}

The proofs of the theorems are collected in the following section. It is worth noticing that we use only elementary means such as basic integration, Jensen's and H\"older's inequalities, and the simplest version of the weighted Hardy inequality.

The approach to the inequalities carried out in Theorem~\ref{T:1} can be applied also to cones of monotone functions. This sort of problems goes beyond the scope of the present paper and we plan to return to it in our forthcoming work.

Throughout, expressions such as $0\cdot \infty$, $\frac{0}{0}$, $\frac{\infty}{\infty}$, $\exp(\log(0))$ etc., are treated as zero.

\section{Proofs}\label{S:proofs}

Let us first recall two useful inequalities that reflect the boundedness of weighted Hardy operator and weighted Copson operator on weighted Lebesgue spaces in a form which suits our purposes. 

\begin{proposition}\label{P:weighted-hardy-copson}
    Assume that $v$ is a weight on $(0,\infty)$ and let $V(t)=\int_{0}^{t}v(s)\,ds$ for $t\in(0,\infty)$. 
    
    \textup{(i)} For every $p\in[1,\infty)$ and $\alpha\in(-\infty,p-1)$, there exists a constant $A_{p,\alpha}$ depending only on $p$ and $\alpha$ such that
    \begin{equation}\label{E:hardy-weighted}
        \left(\int_{0}^{\infty} \left(\frac{1}{V(t)}\int_{0}^{t}h(s)v(s)\,ds\right)^{p} V(t)^{\alpha}v(t)\,dt\right)^{\frac{1}{p}}
        \le A_{p,\alpha} \left(\int_{0}^{\infty}h(t)^{p}V(t)^{\alpha}v(t)\,dt\right)^{\frac{1}{p}}.
    \end{equation}
    
    \textup{(ii)} For every $p\in[1,\infty)$ and $\alpha\in(-1,\infty)$, there exists a constant $B_{p,\alpha}$ depending only on $p$ and $\alpha$ such that
    \begin{equation}\label{E:copson-weighted}
        \left(\int_{0}^{\infty} \left(\int_{t}^{\infty}h(s)\frac{v(s)}{V(s)}\,ds\right)^{p} V(t)^{\alpha}v(t)\,dt\right)^{\frac{1}{p}}
        \le B_{p,\alpha} \left(\int_{0}^{\infty}h(t)^{p}V(t)^{\alpha}v(t)\,dt\right)^{\frac{1}{p}}
    \end{equation}
    for every $h\in\M_+(0,\infty)$.
\end{proposition}
We omit the proof of Proposition~\ref{P:weighted-hardy-copson}, because the result is well known. For $p\in(1,\infty)$, see e.g.~\cite[Theorem 4 and its dual version]{Kuf:07}, for $p=1$ it follows by a simple use of Fubini's theorem. It should be noted that the constants in Proposition~\ref{P:weighted-hardy-copson} do not depend on $v$.

For $\gamma\in(0,\infty)$, denote 
\begin{equation}\label{E:dgamma}
    D_{\gamma}=\max\left\{1,2^{\gamma-1}\right\}.
\end{equation}
Then elementary analysis shows that
\begin{equation}\label{E:dgamma-inequality}
    (a+b)^{\gamma}\le D_{\gamma}\left(a^{\gamma}+b^{\gamma}\right)
    \quad\text{for every $a,b\in[0,\infty)$.}
\end{equation}
It follows from~\eqref{E:weak-lattice} and~\eqref{E:quasitriangle} that if $\varrho$ is a sub-monotone functional, then, for every $f\in \mathcal M_+(0,\infty)$ and $c,\lambda>0$, one has
\begin{equation}\label{E:general-lambda}
    \vr(cf+\lambda) \le K\lambda \vr(\tfrac{c}{\lambda}f+1)
    \le K^2\lambda\vr(\tfrac{c}{\lambda}f)+K^2\lambda\vr(1)
    \le K^3c\vr(f)+K^2\lambda\vr(1).
\end{equation}

\begin{remark}\label{R:A}
    The inequalities in the statements (vi)--(ix) of Theorem~\ref{T:1} are restricted to strictly positive functions. We would like to point out here that this is not really a significant restriction. For instance, suppose that, for a given function $f$ in $\M_+(0,\infty)$, there exists some $A\in(0,\infty)$ such that $f$ is strictly positive on $(0,A)$, and possibly zero on $[A,\infty)$. Then, given $\varepsilon>0$, we can find a strictly positive function $f_{A,\varepsilon}$ on $(0,\infty)$, such that
    \begin{equation}\label{E:f-A}
        \left(\int_{0}^{\infty}f_{A,\varepsilon}(t)^{p}v(t)\,dt\right)^{\frac{1}{p}}
        \le (1+\varepsilon)
        \left(\int_{0}^{\infty}f(t)^{p}v(t)\,dt\right)^{\frac{1}{p}}.
    \end{equation}
    As a consequence, the inequality~\eqref{E:4} can be equivalently rephrased as    
    \begin{equation}\label{E:4-A}
        \vr\left(\chi_{(0,A)}(t)\exp\left(\frac{1}{V(t)}\int_{0}^{t}\log(f(s))v(s)\,ds\right)\right)
        \le C_6
        \left(\int_{0}^{A}f(t)^{p}v(t)\,dt\right)^{\frac{1}{p}}
    \end{equation}
    for functions that are strictly positive only on $(0,A)$
    (without even changing the constant in the inequality). A similar approach also applies to inequalities in statements (vii)--(ix).
\end{remark}

\begin{proof}[Proof of Theorem~\ref{T:1}]
    (i)$\Rightarrow$(vi). By Jensen's inequality applied to the function $s\mapsto \exp(s)$, which is convex on $(0,\infty)$, we obtain
    \begin{equation*}
        \exp\left(\frac{1}{V(t)}\int_{0}^{t}\log(f(s))v(s)\,ds\right)\le \frac{1}{V(t)}\int_{0}^{t}f(s)v(s)\,ds
    \end{equation*}
    for every $f\in\M_+(0,\infty)$ and every $t\in(0,\infty)$. Thus, by~\eqref{E:lattice} and~\eqref{E:1},
    \begin{align*}
        \vr&\left(\exp\left(\frac{1}{V(t)}\int_{0}^{t}\log(f(s))v(s)\,ds\right) \right)
        \le \vr\left(\frac{1}{V(t)}\int_{0}^{t}f(s)v(s)\,ds\right)
            \\
        &  \le C_1 \left(\int_{0}^{\infty} f(s)^{p}v(s)\,ds\right)^{\frac{1}{p}},
    \end{align*}
    and~\eqref{E:4} follows with $C_6\le C_1$.
    
    (vi)$\Rightarrow$(ii). Fix $r\in[1,\infty)$, $\alpha\in (\max\{-\frac 1p, -\frac1{p'}\},\infty)$, $h\in\M_+(0,\infty)$ and $t\in(0,\infty)$. Then, by properties of $\log$, monotonicity and elementary integration,
    \begin{align*}
        & \int_{0}^{t}\log\left(\left(\int_{s}^{\infty}h(\tau)\frac{v(\tau)}{V(\tau)}\,d\tau\right)^{\frac{r}{p}}V(s)^{\alpha }\right)v(s)\,ds
            \\
        & = \int_{0}^{t}\log\left(\left(\int_{s}^{\infty}h(\tau)\frac{v(\tau)}{V(\tau)}\,d\tau\right)^{\frac{r}{p}}\right)v(s)\,ds
            + \alpha\int_{0}^{t}\log\left(V(s)\right)v(s)\,ds
            \\
        & \ge \log\left(\left(\int_{t}^{\infty}h(\tau)\frac{v(\tau)}{V(\tau)}\,d\tau\right)^{\frac{r}{p}}\right)V(t)
            + \alpha V(t)\left(\log V(t)-1\right).
    \end{align*}
    Dividing this by $V(t)$ and applying the exponential function, we get
    \begin{align*}
        \exp &\left(\frac{1}{V(t)}\int_{0}^{t}\log\left(\left(\int_{s}^{\infty}h(\tau)\frac{v(\tau)}{V(\tau)}\,d\tau\right)^{\frac{r}{p}}V(s)^{\alpha }\right)v(s)\,ds\right)
            \\
        & \ge e^{-\alpha}\left(\int_{t}^{\infty}h(\tau)\frac{v(\tau)}{V(\tau)}\,d\tau\right)^{\frac{r}{p}}V(t)^{\alpha}.
    \end{align*}
    So, calling $\varrho$ into play, using~\eqref{E:lattice} and~\eqref{E:weak-lattice}, and applying (vi) to the function $f$ defined by
    \begin{equation*}
        f(t) = \left(\int_{t}^{\infty}h(\tau)\frac{v(\tau)}{V(\tau)}\,d\tau\right)^{\frac{r}{p}}V(t)^{\alpha }
        \quad\text{for $t\in(0,\infty)$,}
    \end{equation*}
    (resorting to Remark~\ref{R:A} if need be) we obtain
    \begin{align*}
        \varrho&\left(\left(\int_{t}^{\infty}h(\tau)\frac{v(\tau)}{V(\tau)}\,d\tau\right)^{\frac{r}{p}} V(t)^\alpha\right)
            \\
        &\le Ke^{\alpha} \varrho\left(\exp \left(\frac{1}{V(t)}\int_{0}^{t}\log\left(\left(\int_{s}^{\infty}h(\tau)\frac{v(\tau)}{V(\tau)}\,d\tau\right)^{\frac{r}{p}}V(s)^{\alpha }\right)v(s)\,ds\right)\right)
            \\
        &\le KC_6 e^{\alpha} \left(\int_{0}^{\infty} \left(\int_{t}^{\infty}h(s)\frac{v(s)}{V(s)}\,ds\right)^{r} V(t)^{\alpha p}v(t)\,dt\right)^{\frac{1}{p}}.
    \end{align*}
    Thanks to the assumption $\alpha>-\frac{1}{p}$ we can apply~\eqref{E:copson-weighted} to the last expression and obtain
    \begin{equation*}
        \varrho\left(\left(\int_{t}^{\infty}h(\tau)\frac{v(\tau)}{V(\tau)}\,d\tau\right)^{\frac{r}{p}} V(t)^\alpha\right)
        \le KC_6 e^{\alpha} B_{r,\alpha p}^{\frac rp}
        \left(\int_{0}^{\infty}h(t)^{r}V(t)^{\alpha p }v(t)\,dt\right)^{\frac{1}{p}},
    \end{equation*}
    which yields~\eqref{E:2} with $C_{2,1}\le KC_6 e^{\alpha} B_{r,\alpha p}^{\frac rp}$. Finally, plugging $f\equiv1$ into~\eqref{E:4}, we get~\eqref{E:constant} with $C_{2,2}\le C_6$, and (ii) follows. Note that, for this implication, the bound $\alpha>-\frac{1}{p'}$ is not necessary.
 
    (ii)$\Rightarrow$(iii).  This implication holds trivially.
   
    (ii)$\Rightarrow$(iv).
    Fix $r\in[1,\infty)$, 
    $\alpha\in (\max\{-\frac 1p,-\frac1{p'}\},\infty)$ and
    $\beta\in(-\frac{1}{r'},\infty)$ such that~\eqref{E:condition-iv} holds. Fix $h\in\M_+(0,\infty)$ such that $\int_{0}^{\infty}h(s)v(s)\,ds<\infty$, and set 
    \begin{equation*}
        f(t)=V(t)^{-\beta -1}\left(\int_{0}^{t} h(s)v(s)\,ds \right)\quad\text{for $t\in(0,\infty)$.}
    \end{equation*} 
    Fix $t\in(0,\infty)$ such that $0<V(t)<\infty$. 
    Then, by monotonicity and elementary integration, we have
    \begin{align}\label{E:lower-bound-iv}
            \left(\int_{t}^{\infty}f(s)\frac{v(s)}{V(s)}\,ds\right)^{\frac{r}{p}}
            &=
            \left(\int_{t}^{\infty}V(s)^{-\beta -1}\left(\int_{0}^{s} h(\tau)v(\tau)\,d\tau \right)\frac{v(s)}{V(s)}\,ds\right)^{\frac{r}{p}}
                \\
            &\ge \left(\int_{t}^{\infty}V(s)^{-\beta -1}\frac{v(s)}{V(s)}\,ds\right)^{\frac{r}{p}}
            \left(\int_{0}^{t}h(s)v(s)\,ds\right)^{\frac{r}{p}}
                \nonumber
                \\
            &=
            \left(\frac{1}{\beta +1}\right)^{\frac{r}{p}}\left(V(t)^{-\beta -1}-V(\infty)^{-\beta -1}\right)^{\frac{r}{p}}
            \left(\int_{0}^{t}h(s)v(s)\,ds\right)^{\frac{r}{p}}.
                \nonumber
    \end{align}
    We now write
    \begin{equation*}
        V(t)^{-\beta-1}
        =
        V(t)^{-\beta-1} - V(\infty)^{-\beta-1}
        + V(\infty)^{-\beta-1},
    \end{equation*}
    and apply~\eqref{E:dgamma-inequality} with $\gamma=\frac{r}{p}$ to the latter expression to get
    \begin{equation}\label{E:C}
        V(t)^{(-\beta-1)\frac{r}{p}}
        \le D_{\frac{r}{p}}
        \left(V(t)^{-\beta-1} - V(\infty)^{-\beta-1}
        \right)^{\frac{r}{p}}
        + D_{\frac{r}{p}}
        V(\infty)^{(-\beta-1)\frac{r}{p}}.
    \end{equation}
    Now, multiplying both sides of~\eqref{E:C} with $\left(\int_0^thv\right)^{\frac{r}{p}}$, and then applying~\eqref{E:lower-bound-iv}, we get
    \begin{align}\label{E:D}
        V(t)^{(-\beta-1)\frac{r}{p}}
        \left(\int_0^t h(s)v(s)\,ds\right)^{\frac{r}{p}}
        &\le
        D_{\frac{r}{p}}
        \left(V(t)^{-\beta-1} - V(\infty)^{-\beta-1}
        \right)^{\frac{r}{p}}
        \left(\int_0^t h(s)v(s)\,ds\right)^{\frac{r}{p}}
            \\
        &\quad
        + D_{\frac{r}{p}}
        V(\infty)^{(-\beta-1)\frac{r}{p}}
        \left(\int_0^t h(s)v(s)\,ds\right)^{\frac{r}{p}}
            \nonumber
            \\
        &\le
        (\beta+1)^{\frac{r}{p}}
        D_{\frac{r}{p}}
        \left(\int_{t}^{\infty}f(s)\frac{v(s)}{V(s)}\,ds\right)^{\frac{r}{p}}
            \nonumber
            \\
        &\quad
        + D_{\frac{r}{p}}
        V(\infty)^{(-\beta-1)\frac{r}{p}}
        \left(\int_0^t h(s)v(s)\,ds\right)^{\frac{r}{p}}.
            \nonumber
    \end{align}
    We shall now deal with the latter term, namely $\left(\int_0^t h(s)v(s)\,ds\right)^{\frac{r}{p}}$. Assume for the time being that $r>1$. Then, by the H\"older inequality, one has
    \begin{align}\label{E:E}
        \int_0^t h(s)v(s)\,ds
        &=
        \int_0^t h(s)
        V(s)^{-\beta+\frac{\alpha p}{r}}
        V(s)^{\beta-\frac{\alpha p}{r}}
        v(s)\,ds
            \\
        &\le
        \left(\int_{0}^{t}h(s)^r
        V(s)^{-\beta r+\alpha p}v(s)\,ds\right)^{\frac{1}{r}} 
        \left(\int_{0}^{t}
        V(s)^{(\beta-\frac{\alpha p}{r})r'}v(s)\,ds\right)^{\frac{1}{r'}}.
            \nonumber
    \end{align}
    Owing to the assumption~\eqref{E:condition-iv}, we have
    \begin{equation*}
        \left(\beta-\frac{\alpha p}{r}\right)r'+1>0.
    \end{equation*}
    Hence, elementary integration yields
    \begin{equation*}
        \int_{0}^{t}V(s)^{(\beta-\frac{\alpha p}{r})r'}v(s)\,ds
        =
        \frac{1}{(\beta-\frac{\alpha p}{r})r'+1}
        V(t)^{(\beta-\frac{\alpha p}{r})r'+1}.
    \end{equation*}
    Raising both sides of the last identity to $\frac{1}{r'}$ and then plugging it back to~\eqref{E:E}, we get
    \begin{align*}
        \int_0^t h(s)v(s)\,ds
       &\le
        \frac{1}{\left[(\beta-\frac{\alpha p}{r})r'+1\right]^{\frac{1}{r'}}}
        V(t)^{\beta-\frac{\alpha p}{r}+\frac{1}{r'}}
        \left(\int_{0}^{t}h(s)^r
        V(s)^{-\beta r+\alpha p}v(s)\,ds\right)^{\frac{1}{r}}.
    \end{align*}
    When $r=1$, then~\eqref{E:condition-iv} implies $\beta>\alpha p$, whence we get, owing to the monotonicity, 
    \begin{align*}
        \int_0^t h(s)v(s)\,ds
        &=
        \int_0^t h(s)V(s)^{-\beta+\alpha p}V(s)^{\beta-\alpha p}v(s)\,ds
        \le
        V(t)^{\beta-\alpha p}\int_0^t h(s)V(s)^{-\beta+\alpha p}v(s)\,ds.
    \end{align*}
    In any case, raising the last two estimates to $\frac{r}{p}$ and setting
    \begin{equation*}
        \kappa=
            \begin{cases}
                \left[(\beta-\frac{\alpha p}{r})r'+1\right]
                ^{-\frac{r-1}{p}},
                &\text{if $r\in(1,\infty)$}
                    \\
                1
                &\text{if $r=1$,}
            \end{cases}
    \end{equation*}
    we get
    \begin{align}\label{E:G}
        \left(\int_0^t h(s)v(s)\,ds\right)^{\frac{r}{p}}
       &\le
        \kappa
        V(t)^{\beta\frac{r}{p}-\alpha+\frac{r-1}{p}}
        \left(\int_{0}^{t}h(s)^r
        V(s)^{-\beta r+\alpha p}v(s)\,ds\right)^{\frac{1}{p}}.
    \end{align}
    Consequently, inserting~\eqref{E:G} into~\eqref{E:D}, we arrive at
    \begin{align*}
        V(t)^{(-\beta-1)\frac{r}{p}}
        &\left(\int_0^t h(s)v(s)\,ds\right)^{\frac{r}{p}}
        \le
        (\beta+1)^{\frac{r}{p}}
        D_{\frac{r}{p}}
        \left(\int_{t}^{\infty}f(s)\frac{v(s)}{V(s)}\,ds\right)^{\frac{r}{p}}
            \\
        &\quad
        + \kappa D_{\frac{r}{p}}
        V(\infty)^{(-\beta-1)\frac{r}{p}}
        V(t)^{\beta\frac{r}{p}-\alpha+\frac{r-1}{p}}
        \left(\int_{0}^{t}h(s)^r
        V(s)^{-\beta r+\alpha p}v(s)\,ds\right)^{\frac{1}{p}}.
    \end{align*}
    Now, since $V(t)\le V(\infty)$ and $\frac{(\beta+1)r-1}{p}>0$ (owing to the assumption $\beta>-\frac{1}{r'}$), we have
    \begin{equation*}
       V(\infty)^{(-\beta-1)\frac{r}{p}}
        V(t)^{\beta\frac{r}{p}+\frac{r-1}{p}}
        =
        \left(\frac{V(t)}{V(\infty)}\right)^{\frac{(\beta+1)r-1}{p}}V(\infty)^{-\frac{1}{p}}
        \le V(\infty)^{-\frac{1}{p}}.
    \end{equation*}
    Therefore,
    \begin{align*}
        V(t)^{(-\beta-1)\frac{r}{p}}
        &\left(\int_0^t h(s)v(s)\,ds\right)^{\frac{r}{p}}
        \le
        (\beta+1)^{\frac{r}{p}}
        D_{\frac{r}{p}}
        \left(\int_{t}^{\infty}f(s)\frac{v(s)}{V(s)}\,ds\right)^{\frac{r}{p}}
            \\
        &\quad
        + \kappa D_{\frac{r}{p}}
        V(t)^{-\alpha}
        V(\infty)^{-\frac{1}{p}}
        \left(\int_{0}^{t}h(s)^r
        V(s)^{-\beta r+\alpha p}v(s)\,ds\right)^{\frac{1}{p}}.
    \end{align*}
    Multiplying the last inequality by $V(t)^{\alpha}$ and extending the upper bound of the last integral to $\infty$, we arrive at
    \begin{align}\label{E:J}
        V(t)^{\alpha-\frac{\beta r}{p}}
        &\left(\frac{1}{V(t)}\int_0^t h(s)v(s)\,ds\right)^{\frac{r}{p}}
        \le
        (\beta+1)^{\frac{r}{p}}
        D_{\frac{r}{p}}
        V(t)^{\alpha}
        \left(\int_{t}^{\infty}f(s)\frac{v(s)}{V(s)}\,ds\right)^{\frac{r}{p}}
            \\
        &\quad
        + \kappa D_{\frac{r}{p}}
        V(\infty)^{-\frac{1}{p}}
        \left(\int_{0}^{\infty}h(s)^r
        V(s)^{-\beta r+\alpha p}v(s)\,ds\right)^{\frac{1}{p}}.
            \nonumber
    \end{align}
    Applying~\eqref{E:general-lambda} to~\eqref{E:J}, we obtain
    \begin{align}\label{E:upper-bound-iv-5}
            &\varrho\left(V(t)^{\alpha-\beta\frac{r}{p}}
            \left(\frac{1}{V(t)}
            \int_{0}^{t}h(s)v(s)\,ds\right)^{\frac{r}{p}}\right)
                \\
          & \qquad\le 
            K^3(\beta+1)^{\frac{r}{p}} D_{\frac{r}{p}}
            \varrho\left( V(t)^{\alpha}
            \left(\int_{t}^{\infty}f(s)\frac{v(s)}{V(s)}\,ds\right)^{\frac{r}{p}}\right)
                \nonumber
                \\
            & \qquad\qquad +
            K^2 D_{\frac{r}{p}}\kappa V(\infty)^{-\frac{1}{p}}
            \left(\int_{0}^{\infty}h(s)^rV(s)^{-\beta r+\alpha p}v(s)\,ds\right)^{\frac{1}{p}}\varrho(\boldsymbol{1}).
                \nonumber
    \end{align}
    By~\eqref{E:constant}, we have
    \begin{equation}\label{E:inter}
        \varrho(\boldsymbol{1}) \le C_{2,2}V(\infty)^{\frac{1}{p}}.
    \end{equation}
    Owing to~\eqref{E:2} and the definition of $f$, one gets
    \begin{align*}
            \varrho\left( V(t)^{\alpha}
            \left(\int_{t}^{\infty}f(s)\frac{v(s)}{V(s)}\,ds\right)^{\frac{r}{p}}\right)
            &\le C_{2,1} 
            \left(\int_{0}^{\infty}f(t)^rV(t)^{\alpha p}v(t)\,dt\right)^{\frac{1}{p}}
                \\
            &= C_{2,1} 
            \left(\int_{0}^{\infty}\left(\int_{0}^{t} h(s)v(s)\,ds \right)^r
            V(t)^{-(\beta+1)r+\alpha p}v(t)\,dt\right)^{\frac{1}{p}}
                \\
            &= C_{2,1} 
            \left(\int_{0}^{\infty}\left(\frac{1}{V(t)}\int_{0}^{t} h(s)v(s)\,ds \right)^r
            V(t)^{-\beta r+\alpha p}v(t)\,dt\right)^{\frac{1}{p}}.
    \end{align*}
    Thanks to~\eqref{E:condition-iv}, we can use~\eqref{E:hardy-weighted} to obtain
    \begin{align}\label{E:upper-bound-iv-6-bis}
            \varrho\left( V(t)^{\alpha}
            \left(\int_{t}^{\infty}f(s)\frac{v(s)}{V(s)}\,ds\right)^{\frac{r}{p}}\right)
            &\le C_{2,1} A_{r,-\beta r+\alpha p}^{\frac{r}{p}}
            \left(\int_{0}^{\infty}h(t)^rV(t)^{-\beta r+\alpha p}v(t)\,dt \right)^{\frac{1}{p}}.
    \end{align}
    Thus, combining~\eqref{E:upper-bound-iv-5}, \eqref{E:inter} and~\eqref{E:upper-bound-iv-6-bis}, one finally has
    \begin{align*}
            &\varrho\left(V(t)^{-\beta\frac{r}{p}+\alpha}\left(\frac{1}{V(t)}
            \int_{0}^{t}h(s)v(s)\,ds\right)^{\frac{r}{p}}\right)
                \\
            &\qquad\le 
            \left(K^3(\beta+1)^{\frac{r}{p}} D_{\frac{r}{p}}C_{2,1}A_{r,-\beta r+\alpha p}^{\frac{r}{p}}
            + K^2C_{2,2} D_{\frac{r}{p}}\kappa\right)
            \left(\int_{0}^{\infty}h(t)^rV(t)^{-\beta r+\alpha p}v(t)\,dt \right)^{\frac{1}{p}},
    \end{align*}
    and~\eqref{E:3} follows with
    \begin{equation*}
        C_4=K^3(\beta+1)^{\frac{r}{p}}D_{\frac{r}{p}} C_{2,1}A_{r,-\beta r+\alpha p}^{\frac{r}{p}}
            + K^2C_{2,2} D_{\frac{r}{p}}\kappa.
    \end{equation*}
   
    (iv)$\Rightarrow$(v). This implication holds trivially.
    
    (v)$\Rightarrow$(iii). Assume that $(r,\alpha,\beta)$ is a triple whose existence is guaranteed by~(v). Fix $h\in\M_+(0,\infty)$ and set
    \begin{equation*}
        f(t)=V(t)^{\beta}\int_{t}^{\infty}h(s)\frac{v(s)}{V(s)}\,ds
        \quad\text{for $t\in(0,\infty)$.}
    \end{equation*}
    By~\eqref{E:copson-weighted}, one has
    \begin{align}\label{E:v-iii-1}
            \left(\int_{0}^{\infty}f(t)^rV(t)^{-\beta r+\alpha p}v(t)\,dt\right)^{\frac{1}{p}}
            &=
            \left(\int_{0}^{\infty}\left(\int_{t}^{\infty}h(s)\frac{v(s)}{V(s)}\,ds\right)^rV(t)^{\alpha p}v(t)\,dt\right)^{\frac{1}{p}}
                \\
            &\le B_{r,\alpha p}^{\frac{r}{p}}
            \left(\int_{0}^{\infty}h(t)^{r}V(t)^{\alpha p}v(t)\,dt\right)^{\frac{1}{p}}.
                \nonumber
        \end{align}
    Fix $t\in(0,\infty)$. Then, owing to monotonicity and using integration, we obtain
    \begin{align}\label{E:v-iii-2}
            \left(\frac{1}{V(t)}\int_{0}^{t}f(s)v(s)\,ds\right)^{\frac{r}{p}}
            &=
            \left(\frac{1}{V(t)}\int_{0}^{t}V(s)^{\beta}
            \int_{s}^{\infty}h(\tau)\frac{v(\tau)}{V(\tau)}\,d\tau
            v(s)\,ds\right)^{\frac{r}{p}}
                \\
            &\ge V(t)^{-\frac{r}{p}}
            \left(\int_{t}^{\infty}h(s)\frac{v(s)}{V(s)}\,ds\right)^{\frac{r}{p}}
            \left(\int_{0}^{t}V(s)^{\beta}v(s)\,ds\right)^{\frac{r}{p}}
                \nonumber
                \\
            &= \frac{1}{(\beta+1)^{\frac{r}{p}}} V(t)^{\frac{\beta r}{p}}
            \left(\int_{t}^{\infty}h(s)\frac{v(s)}{V(s)}\,ds\right)^{\frac{r}{p}}.
                \nonumber
    \end{align}
    Thus, by~\eqref{E:v-iii-2}, \eqref{E:lattice}, \eqref{E:weak-lattice}, \eqref{E:3-bis}, and~\eqref{E:v-iii-1}, we get
    \begin{align*}
            \varrho\left(V(t)^{\alpha}\left(\int_{t}^{\infty}h(s)\frac{v(s)}{V(s)}\,ds\right)^{\frac{r}{p}}\right)
            &\le
            \varrho\left((\beta+1)^{\frac{r}{p}}V(t)^{\alpha-\frac{\beta r}{p}}\left(\frac{1}{V(t)}\int_{0}^{t}f(s)v(s)\,ds\right)^{\frac{r}{p}}\right)
                \\
            &\le KC_5(\beta+1)^{\frac{r}{p}}
            \left(\int_{0}^{\infty}f(t)^rV(t)^{-\beta r+\alpha p}v(t)\,dt\right)^{\frac{1}{p}}
                \\
            &\le KC_5(\beta+1)^{\frac{r}{p}}B_{r,\alpha p}^{\frac{r}{p}}
            \left(\int_{0}^{\infty}h(t)^{r}V(t)^{\alpha p}v(t)\,dt\right)^{\frac{1}{p}},
    \end{align*}
    and~\eqref{E:2-bis} follows with
    \begin{equation*}
            C_3\le KC_5(\beta+1)^{\frac{r}{p}}B_{r,\alpha p}^{\frac{r}{p}}.
    \end{equation*}
    To prove~\eqref{E:constant}, denote
    \begin{equation*}
        c=\beta-\frac{\alpha p}{r}+1,
    \end{equation*}
    and note that~\eqref{E:condition-iv} implies that $c>0$. Set
    \begin{equation*}
        f(t)=cV(t)^{\beta-\frac{\alpha p}{r}}\quad\text{for $t\in(0,\infty)$.}
    \end{equation*}
    Then
    \begin{equation*}
        \int_{0}^{t}f(s)v(s)\,ds = c \int_{0}^{t}V(s)^{\beta-\frac{\alpha p}{r}}v(s)\,ds = V(t)^{\beta-\frac{\alpha p}{r}+1}\quad\text{for $t\in(0,\infty)$,}
    \end{equation*}
    hence~\eqref{E:3-bis} implies that
    \begin{equation*}
        \varrho(\boldsymbol{1}) \le C_5 c^{\frac{r}{p}}V(\infty)^{\frac{1}{p}},
    \end{equation*}
    which is~\eqref{E:constant} with $C_{2,2}\le C_5 c^{\frac{r}{p}}$. This establishes~(iii).
    
    (vi)$\Rightarrow$ (vii) Fix $r\in(0,\infty)$, a measurable strictly positive function $g$ on $(0,\infty)$ and $t\in(0,\infty)$. Then, by Jensen's inequality, applied with the function $s\mapsto \exp(\frac{s}{r})$, which is convex on $(0,\infty)$, we obtain 
    \begin{equation*}
        \exp\left(\frac{1}{V(t)}\int_{0}^{t}\log(g(s)^r)v(s)\,ds\right)
            \le \left(\frac{1}{V(t)}\int_{0}^{t}g(s)v(s)\,ds\right)^{r},
    \end{equation*}
    or, which is the same,
    \begin{equation}\label{E:vi-vii-1}
        \left(\frac{V(t)}{\int_{0}^{t}g(s)v(s)\,ds}\right)^{r}
        \le
        \frac{1}{\exp\left(\frac{1}{V(t)}\int_{0}^{t}\log(g(s)^r)v(s)\,ds\right)}.
    \end{equation}
    Now let $f$ be a measurable strictly positive function on $(0,\infty)$. We set $g=\frac{1}{f}$. Then $g$ is also a measurable strictly positive function on $(0,\infty)$, hence it follows from~\eqref{E:vi-vii-1} that
    \begin{equation}\label{E:vi-vii-2}
        \left(\frac{V(t)}{\int_{0}^{t}f(s)^{-1}v(s)\,ds}\right)^{r}
        \le
        \exp\left(\frac{1}{V(t)}\int_{0}^{t}\log(f(s)^r)v(s)\,ds\right)
        \quad\text{for every $t\in(0,\infty)$.}
    \end{equation}
    Consequently, combining~\eqref{E:vi-vii-2}, \eqref{E:lattice}, and~\eqref{E:4} applied to $f^r$ in place of $f$, we obtain~\eqref{E:5} with $C_7\le C_6$.

    (vii)$\Rightarrow$(viii). This implication holds trivially.
    
    (viii)$\Rightarrow$(iii). Fix $h\in\M_+(0,\infty)$ and let $r\in(0,\infty)$ be the parameter (whose existence is guaranteed by (viii)) such that~\eqref{E:5-bis} holds. Set
    \begin{equation*}
        f(t) = 
        \left(\int_{t}^{\infty}h(s)\frac{v(s)}{V(s)}\,ds\right)^{\frac{1}{pr}}            \quad\text{for $t\in(0,\infty)$.}
        \end{equation*}
    Assume first that $f$ is finite and strictly positive on $(0,\infty)$. Fix $t\in(0,\infty)$. Owing to the fact that $f$ is nonincreasing, one has
    \begin{align*}
      \left(\int_{0}^{t}f(s)^{-1}v(s)\,ds\right)^{-r}
      &\ge f(t)^{r}V(t)^{-r}
      = \left(\int_{t}^{\infty}h(s)\frac{v(s)}{V(s)}\,ds\right)^{\frac{1}{p}}
      V(t)^{-r},
    \end{align*}
    that is,
    \begin{equation*}
        \left(\int_{t}^{\infty}h(s)\frac{v(s)}{V(s)}\,ds\right)^{\frac{1}{p}}
        \le
        \left(\frac{V(t)}{\int_{0}^{t}f(s)^{-1}v(s)\,ds}\right)^{r}.
    \end{equation*}
    Therefore, owing to~\eqref{E:lattice} and~\eqref{E:5-bis}, we get, by Fubini's theorem,
    \begin{equation*}
        \begin{split}
            \varrho\left(\left(\int_{t}^{\infty}h(s)\frac{v(s)}{V(s)}\,ds\right)^{\frac{1}{p}}\right)
            &\le C_8
            \left(\int_{0}^{\infty}f(t)^{rp}v(t)\,dt\right)^{\frac{1}{p}}
            = C_8
            \left(\int_{0}^{\infty}v(t)\int_{t}^{\infty}h(s)\frac{v(s)}{V(s)}\,ds\,dt\right)^{\frac{1}{p}}
                \\
            &= C_8
            \left(\int_{0}^{\infty}h(s)v(s)\,ds\right)^{\frac{1}{p}},
            \end{split}
   \end{equation*}
   and~\eqref{E:2-bis} (with $r=1$ and $\alpha=0$) follows with $C_3\le C_8$. Moreover, setting $f\equiv 1$ and plugging it into~\eqref{E:5-bis},
   we immediately obtain~\eqref{E:constant} with $C_{2,2}\le C_8$. This establishes (iii) for finite and strictly positive $f$. In the case of general $f$, the assertion follows similarly using Remark~\ref{R:A}.
    
    (i) $\Rightarrow$ (ix) Fix $f\in\M_+(0,\infty)$ and $t\in(0,\infty)$. It follows from the assumptions of (ix) that the function $\varphi^{-1}$ is convex. Owing to the Jensen inequality, we get
    \begin{equation*}
        \varphi^{-1}\left(\frac{1}{V(t)}\int_{0}^{t}\varphi(f(s))v(s)\,ds\right)
        \le \frac{1}{V(t)}\int_{0}^{t}f(s)v(s)\,ds.
    \end{equation*}
    Thus, (ix) with $C_9\le C_1$ follows immediately from (i) via~\eqref{E:lattice}.
    
    (ix)$\Rightarrow$(iii).
    Fix $h\in\M_+(0,\infty)$ and set
    \begin{equation*}
        f(t) = \int_{t}^{\infty}h(s)\frac{v(s)}{V(s)}\,ds\quad\text{for $t\in(0,\infty)$.}
    \end{equation*}
    Then $f$ is non-increasing, hence the assumptions on $\varphi$ yield 
    \begin{equation*}
         \varphi^{-1}\left(\frac{1}{V(t)}\int_{0}^{t}\varphi(f(s))v(s)\,ds\right)
         \ge
         \varphi^{-1}\left(\varphi(f(t))\frac{1}{V(t)}\int_{0}^{t}v(s)\,ds\right) =f(t)=
     \int_{t}^{\infty}h(s)\frac{v(s)}{V(s)}\,ds.    
    \end{equation*}
    Thus, by~\eqref{E:9-bis},~\eqref{E:lattice} and~\eqref{E:copson-weighted}
    \begin{equation*}
        \begin{split}
            \varrho\left(\int_{t}^{\infty}h(s)\frac{v(s)}{V(s)}\,ds\right)
            &\le
            \varrho\left(\varphi^{-1}\left(\frac{1}{V(t)}\int_{0}^{t}\varphi(f(s))v(s)\,ds\right)\right)
               \le C_9
            \left(\int_{0}^{\infty}f(t)^{p}v(t)\,dt\right)^{\frac{1}{p}}
                \\
            &= C_9
            \left(\int_{0}^{\infty}\left(\int_{t}^{\infty}h(s)\frac{v(s)}{V(s)}\,ds\right)^{p}v(t)\,dt\right)^{\frac{1}{p}}
                \le C_9B_{p,0}
            \left(\int_{0}^{\infty}h(t)^{p}v(t)\,dt\right)^{\frac{1}{p}},
        \end{split}
    \end{equation*}
    which is~\eqref{E:2-bis} with $\alpha=0$, $r=p$ and $C_3\le C_9$. Moreover, setting $f\equiv 1$ and plugging it into~\eqref{E:5-bis},
   we immediately obtain~\eqref{E:constant} with $C_{2,2}\le C_8$. This establishes (iii) (with the usual modification using Remark~\ref{R:A} when necessary).

        (iii)$\Rightarrow$(i). Let $r\in[1,\infty)$ and $\alpha\in(\max\{-\frac{1}{p},-\frac{1}{p'}\}, \infty)$ be the parameters, whose existence is guaranteed by (iii), such that~\eqref{E:2-bis} holds. Fix $t\in(0,\infty)$. We first observe that
    \begin{align}\label{E:formula-for-V}
            \frac{1}{V(t)}&=V(t)^{\alpha}[V(t)^{-\frac{(\alpha+1)p}{r}}]^{\frac rp}
                =V(t)^{\alpha}[V(t)^{-\frac{(\alpha+1)p}{r}}-V(\infty)^{-\frac{(\alpha+1)p}{r}}+V(\infty)^{-\frac{(\alpha+1)p}{r}}]^{\frac rp}
                \\
            &\le D_{\frac{r}{p}}V(t)^{\alpha}\left([V(t)^{-\frac{(\alpha+1)p}{r}}-V(\infty)^{-\frac{(\alpha+1)p}{r}}]^{\frac rp}+V(\infty)^{-\alpha-1}\right),
                \nonumber
    \end{align}
    where $D_{\frac{r}{p}}$ is from~\eqref{E:dgamma}. Now fix $h\in\M_+(0,\infty)$ such that 
    \begin{equation}\label{E:condition-finiteness}
        \int_{0}^{\infty}h(s)^{p}v(s)\,ds<\infty.
    \end{equation}
    Then, by the H\"older inequality, one clearly has
    \begin{equation*}%\label{E:condition-finiteness}
        \int_{0}^{t}h(s)v(s)\,ds<\infty
        \quad\text{for every $t\in(0,\infty)$}
    \end{equation*}
    owing to the fact that $V(t)<\infty$ for every $t\in(0,\infty)$.
    Thus~\eqref{E:formula-for-V} yields
    \begin{align}\label{E:last-1}
            &\frac{1}{V(t)}\int_{0}^{t}h(s)v(s)\,ds
                \\
            &\le
            D_{\frac{r}{p}}V(t)^{\alpha} [V(t)^{-\frac{(\alpha+1)p}{r}}-V(\infty)^{-\frac{(\alpha+1)p}{r}}]^{\frac rp}\int_{0}^{t}h(s)v(s)\,ds
                + D_{\frac{r}{p}}V(\infty)^{-\alpha-1}V(t)^{\alpha}\int_{0}^{t}h(s)v(s)\,ds
                \nonumber
                \\
            &= I+II.
                \nonumber
    \end{align}
    Using H\"older's inequality and extending the integration region to $(0,\infty)$, one has
    \begin{equation*}
        \int_{0}^{t}h(s)v(s)\,ds
        \le
        V(t)^{\frac{1}{p'}}\left(\int_{0}^{t}h(s)^pv(s)\,ds\right)^{\frac{1}{p}}
        \le
        V(t)^{\frac{1}{p'}}\left(\int_{0}^{\infty}h(s)^pv(s)\,ds\right)^{\frac{1}{p}},
    \end{equation*}
    whence, owing to the monotonicity of $V$ and the assumption $\alpha+\frac{1}{p'}>0$, we get
    \begin{align}\label{E:II}
        II &\le D_{\frac{r}{p}}\left(\frac{V(t)}{V(\infty)}\right)^{\alpha+\frac{1}{p'}}V(\infty)^{-\frac{1}{p}}\left(\int_{0}^{\infty}h(s)^pv(s)\,ds\right)^{\frac{1}{p}}
            \le D_{\frac{r}{p}}V(\infty)^{-\frac{1}{p}}\left(\int_{0}^{\infty}h(s)^pv(s)\,ds\right)^{\frac{1}{p}}.
    \end{align}
    In order to estimate $I$, define the function $f$ by
    \begin{equation*}
        f(s)=V(s)^{-\frac{(\alpha+1)p}{r}}\left(\int_{0}^{s}h(\tau)v(\tau)\,d\tau\right)^{\frac{p}{r}}\quad\text{for $s\in(0,\infty)$,}
    \end{equation*}
    and note that, by monotonicity, 
    \begin{align*}
            \left(\int_{t}^{\infty}f(s)\frac{v(s)}{V(s)}\,ds\right)^{\frac{r}{p}}
            &=\left(\int_{t}^{\infty}V(s)^{-\frac{(\alpha+1)p}{r}-1}\left(\int_{0}^{s}h(\tau)v(\tau)\,d\tau\right)^{\frac{p}{r}}v(s)\,ds\right)^{\frac{r}{p}}
                \\
            &\ge \left(\int_{t}^{\infty}V(s)^{-\frac{(\alpha+1)p}{r}-1}v(s)\,ds\right)^{\frac{r}{p}}\int_{0}^{t}h(s)v(s)\,ds
                \\
            &= \frac{r}{(\alpha +1)p}\left[V(t)^{\frac{-(\alpha+1)p}{r}}- V(\infty)^{\frac{-(\alpha+1)p}{r}}\right]^{\frac{r}{p}}\int_{0}^{t}h(s)v(s)\,ds.
    \end{align*}
    Consequently,
    \begin{equation}\label{E:one}
        I \le \frac{(\alpha+1)p}{r}D_{\frac{r}{p}} V(t)^{\alpha} \left(\int_{t}^{\infty}f(s)\frac{v(s)}{V(s)}\,ds\right)^{\frac{r}{p}}.
    \end{equation}
    Altogether,~\eqref{E:last-1}, \eqref{E:II} and~\eqref{E:one}, we get 
    \begin{equation}\label{E:estimate-I+II}
        \frac{1}{V(t)}\int_{0}^{t}h(s)v(s)\,ds \le 
        \frac{(\alpha+1)p}{r}D_{\frac{r}{p}} V(t)^{\alpha} \left(\int_{t}^{\infty}f(s)\frac{v(s)}{V(s)}\,ds\right)^{\frac{r}{p}} + D_{\frac{r}{p}}V(\infty)^{-\frac{1}{p}}\left(\int_{0}^{\infty}h(s)^pv(s)\,ds\right)^{\frac{1}{p}}.
    \end{equation}
    As~\eqref{E:estimate-I+II} is valid pointwise for every $t\in(0,\infty)$, calling $\varrho$ into play and using~\eqref{E:general-lambda} followed by~\eqref{E:weak-lattice}, we obtain
    \begin{align}\label{E:varrho-1}
            \varrho\left(\frac{1}{V(t)}\int_{0}^{t}h(s)v(s)\,ds\right)
            &\le K^3\frac{(\alpha+1)p}{r}D_{\frac{r}{p}}
            \varrho\left(V(t)^{\alpha} \left(\int_{t}^{\infty}f(s)\frac{v(s)}{V(s)}\,ds\right)^{\frac{r}{p}}\right)
                \\
            &\qquad\qquad+ K^2D_{\frac{r}{p}}V(\infty)^{-\frac{1}{p}}\left(\int_{0}^{\infty}h(s)^pv(s)\,ds\right)^{\frac{1}{p}}\vr(\boldsymbol{1}).
                \nonumber
    \end{align}
    Now, (iii) implies that
    \begin{equation}\label{E:varrho-2}
        \varrho\left(V(t)^{\alpha} \left(\int_{t}^{\infty}f(s)\frac{v(s)}{V(s)}\,ds\right)^{\frac{r}{p}}\right)
        \le C_3 \left(\int_{0}^{\infty}f(t)^r V(t)^{\alpha p } v(t)\,dt\right)^{\frac{1}{p}}
    \end{equation}
    and
    \begin{equation}\label{E:varrho-3}
        \varrho(\boldsymbol{1}) \le C_{2,2}V(\infty)^{\frac{1}{p}}.
    \end{equation}
    Plugging~\eqref{E:varrho-2} and~\eqref{E:varrho-3} into~\eqref{E:varrho-1}, we infer that
    \begin{align}\label{E:last-but-one}
            \varrho\left(\frac{1}{V(t)}\int_{0}^{t}h(s)v(s)\,ds\right)
                &\le C_3K^3\frac{(\alpha+1)p}{r}D_{\frac{r}{p}}
            \left(\int_{0}^{\infty}f(t)^r V(t)^{\alpha p } v(t)\,dt\right)^{\frac{1}{p}}
                \\
            &\qquad+ C_{2,2}K^2D_{\frac{r}{p}}\left(\int_{0}^{\infty}h(s)^pv(s)\,ds\right)^{\frac{1}{p}}.
                \nonumber
    \end{align}
    Owing to~\eqref{E:hardy-weighted}, applied to $\alpha=0$, and the definition of $f$, we have
    \begin{align}\label{E:hardy-inside}
            \left(\int_{0}^{\infty}f(t)^r V(t)^{\alpha p} v(t)\,dt\right)^{\frac{1}{p}}
            & = \left(\int_0^\infty \left(\frac{1}{V(t)}\int_{0}^{t}h(s)v(s)\,ds\right)^p v(t)\,dt\right)^{\frac1 p}
                \\
            & \le A_{p,0}\left(\int_0^\infty h(t)^p v(t)\,dt\right)^{\frac 1p}.
                \nonumber
    \end{align}
    Thus, finally, by~\eqref{E:last-but-one} and~\eqref{E:hardy-inside}, one has
    \begin{equation*}
            \varrho\left(\frac{1}{V(t)}\int_{0}^{t}h(s)v(s)\,ds\right)
                \le
                \left[C_3K^3
                \frac{(\alpha+1)p}{r}A_{p,0}
                D_{\frac{r}{p}}
                + C_{2,2}K^2D_{\frac{r}{p}}\right]\left(\int_{0}^{\infty}h(s)^pv(s)\,ds\right)^{\frac{1}{p}},
    \end{equation*}
    which is~(i) with
    \begin{equation*}
        C_1 \le C_3K^3D_{\frac{r}{p}}\frac{(\alpha+1)p}{r}A_{p,0}
         + C_{2,2}K^2D_{\frac{r}{p}}.
    \end{equation*}
    This shows the assertion for every $h$ satisfying~\eqref{E:condition-finiteness}. If this condition is not satisfied, then (i) holds trivially. 
    The proof is complete.
\end{proof}

\begin{proof}[Proof of Theorem~\ref{T:3}]
    We rewrite the inequality \eqref{E:10} as
     \begin{equation*}
         \left(\vr\left(\left(\exp\left( \frac{1}{U(x)}\int_{0}^{x}\log(f)(t)u(t)\,dt\right)\right)^m w(x)\right)\right)^{\frac{1}{m}}
        \le C_{10}^q \left(\int_{0}^{\infty}f(x)^{mp}u(x)\,dx\right)^{\frac{1}{mp}}.
     \end{equation*}
    Using Theorem~\ref{T:1} for $\vr_{m,w}(f)=\vr(f^m w)^{\frac{1}{m}}$ instead of $\vr$, and for $v=u$ and $p=mp$, we obtain 
    the claim.
\end{proof}

\begin{proof}[Proof of Theorem~\ref{T:4}]
    We rewrite the inequality \eqref{E:13} as
    \begin{equation*}
        \vr_m\left(\left(\frac{\widetilde{U}(x)}{\int_{0}^{x} f(t)^{-1}
        \widetilde{u}\,dt}\right)^{\frac{1}{m}}\right)
        \le C_{17} \left(\int_{0}^{\infty}f(x)^{mp\frac{1}{m}}   \widetilde{u}(x)\,dx\right)^{\frac{1}{pm}},
    \end{equation*}
    where $\vr_m(g)=\vr\left( g(x)^m U(x) \widetilde{U}(x)^{-1} \right)^{\frac{1}{m}} $.
    We then use Theorem~\ref{T:1} with $r=\frac{1}{m}$ and $p=mp$, and the claim follows.
\end{proof}

\section*{Compliance with Ethical Standards}\label{conflicts}

\smallskip
\par\noindent

{\bf Funding}. This research was partly funded by:
\\ (i) Czech Science Foundation,  grant number 23-04720S 
(A.~Gogatishvili, L. Pick);
\\ (ii) Czech Academy of Sciences, grant number RVO:679858 (A.~Gogatishvili);
\\ (iii) Shota Rustaveli National Science Foundation of Georgia (SRNSFG), grant no FR21-12353 (A.~Gogati\-shvili).

\bigskip
\par\noindent
{\bf Conflict of Interest}. The authors declare that they have no conflict of interest.

\bigskip
\par\noindent
%{\bf Acknowledgment}. We would like to thank the referees for critical reading of the paper and for many very useful suggestions which led to a significant improvement of the quality of the final version of the paper.

\

%\bibliography{AMIRAN-geometric}

\ 

\end{document}